\documentclass[11pt]{amsart}

\usepackage{latexsym}
\usepackage{amssymb}
\usepackage{amsmath}
\usepackage{amsfonts}
\usepackage[mathscr]{eucal}

\usepackage{amscd}

\usepackage{psfrag}

\usepackage{amsthm}
\usepackage[all,cmtip]{xy}
\usepackage{hyperref}
\usepackage{comment}

\input{xypic}

\usepackage{tikz}
\usetikzlibrary{backgrounds,calc}
\usetikzlibrary{calc,arrows}

\usepackage{lineno}

\newtheorem{thm}{Theorem}[section]

\newtheorem{lem}[thm]{Lemma}
\newtheorem{prop}[thm]{Proposition}

\theoremstyle{definition}

\newtheorem{example}[thm]{Example}
\newtheorem{rem}[thm]{Remark}

\newtheorem*{ack}{Acknowledgements}

\numberwithin{equation}{section}

\newcommand{\RR}{\mathbb{R}}

\DeclareMathOperator{\Iso}{Isom}

\DeclareMathOperator{\inte}{int}
\DeclareMathOperator{\img}{Image}

\newcommand{\sphere}{\mathrm{\mathbb{S}}}

\newcommand{\SO}{\mathrm{SO}}

\newcommand{\Id}{\mathsf{1}}

\begin{document}



\title[Isometry groups of Alexandrov spaces]{Isometry groups of Alexandrov spaces}


\author[F. Galaz-Garcia]{Fernando Galaz-Garcia*}
\address[Galaz-Garcia]{Mathematisches Institut, WWU M\"unster, Germany.}
\email{f.galaz-garcia@uni-muenster.de}
\thanks{*The author is part of  SFB 878: \emph{Groups, Geometry \& Actions}, at the University of M\"unster.}


\author[L. Guijarro]{Luis Guijarro**}
\address[Guijarro]{ Department of Mathematics, Universidad Aut\'onoma de Madrid, and ICMAT CSIC-UAM-UCM-UC3M, Spain.}
\curraddr{}
\email{luis.guijarro@uam.es}
\thanks{**Supported by research grants MTM2008-02676,  MTM2011-22612 from the Ministerio de Ciencia e Innovaci\'on (MCINN) and MINECO: ICMAT Severo Ochoa project SEV-2011-0087}

\date{\today}


\subjclass[2000]{53C23, 53C20}
\keywords{Alexandrov space, isometry group, boundary, symmetric space, locally symmetric space}


\begin{abstract}
Let $X$ be an Alexandrov space (with curvature bounded below). We determine the maximal dimension of the isometry group $\Iso(X)$ of $X$ and show that $X$ is isometric to a Riemannian manifold, provided the dimension of $\Iso(X)$ is maximal. We determine gaps in the possible dimensions of $\Iso(X)$. We determine the maximal dimension of $\Iso(X)$ when the boundary $\partial X$ is non-empty and classify up to homeomorphism Alexandrov spaces with boundary and isometry group of maximal dimension. We also show that a symmetric Alexandrov space is isometric to a  Riemannian manifold. 
\end{abstract}

\maketitle




\section{Introduction}
Alexandrov spaces (with a lower curvature bound) are synthetic generalizations of Riemannian manifolds with a lower (sectional) curvature bound. These spaces  arise naturally as Gromov-Hausdorff limits of sequences of Riemannian $n$-manifolds with a uniform lower curvature bound or as orbit spaces of  isometric group actions on Riemannian manifolds with curvature bounded below. Alexandrov spaces have provided a useful tool in the study of smooth and Riemannian manifolds (see, for example, \cite{Gr1,Gr2,Pe2,Pe3,Pe4,SY,Wi}). 

Since the class of Alexandrov spaces properly contains the class of Riemannian manifolds, it is natural to ask to what extent  Riemannian geometry can be generalized to the Alexandrov setting. This problem has received significant attention, especially from an analytic point of view (see, for example, \cite{KMS} and references therein). In this paper we adopt the viewpoint of transformation groups and focus our attention on isometries of  Alexandrov spaces, inspired by the fact that in the Riemannian case the isometry group has been extensively studied (see, for example, \cite{Ko} and related bibliography).  In \cite{FY}, Fukaya and Yamaguchi showed that, as in the case of Riemannian manifolds (see \cite{MS}), the isometry group of an Alexandrov space is a Lie group. By a theorem of van Dantzig and van der Waerden \cite{vDvW} (cf.~Theorem~1.1 in \cite[ Chapter II]{Ko}), the isometry group of a compact Alexandrov space is compact.  The presence of an isometric action has been used to obtain further information on the structure of Alexandrov spaces. In  \cite{Ber}, Berestovski{\u\i} considered finite dimensional homogeneous metric spaces and showed that when they have a lower curvature bound, they are also smooth manifolds (see also \cite{BePl}). The structure of Alexandrov spaces of  cohomogeneity one and their classification in low dimensions appears in \cite{GaSe}. 

The main difference  between isometries of Alexandrov spaces and isometries of Riemannian manifolds is that the former must preserve the metric singularities of the space, i.e., they must map extremal sets to extremal sets, points with a given type of space of directions to others of the same type, etc. This not only imposes severe restrictions on the possible isometries but, reciprocally, the existence of isometries usually results in the appearance of additional structure in the singular sets of an Alexandrov space. 

The present paper explores this theme in several directions, extending various results on isometry groups of Riemannian manifolds to the Alexandrov case. We start with section~\ref{S:Prelims}, where we have collected some background material, along with Alexandrov versions of the Isotropy Lemma and the Principal Isotropy Theorem. In section~\ref{S:Bounds} we derive an upper bound on the dimension of the isometry group of an Alexandrov space of dimension $n$; this bound is the same as the one for Riemannian $n$-manifolds (cf.~\cite{Ko}), which is natural, since the existence of singular points in an Alexandrov space should mean fewer isometries than in the Riemannian case, as pointed out above. We make this more explicit in section~\ref{S:Rigidity}, where we prove that the maximal dimension for the isometry group of an Alexandrov space is only attained in the Riemannian case. The results in sections~\ref{S:Bounds} and \ref{S:Rigidity} were proved  for Riemannian orbifolds, a special class of Alexandrov spaces, in \cite{BZ}. 
 In section~\ref{S:bounds_extremal} we bound the dimension of the isometry group of an Alexandrov space when there are extremal subsets. In section~\ref{S:Gap_thms} we show  that the gaps in the possible dimensions of the isometry group of a Riemannian manifold derived in \cite{Ma} also occur for Alexandrov spaces.
  In section~\ref{S:Bdry} we bound the dimension of the isometry group of an Alexandrov space with boundary and classify these spaces up to homeomorphism when the isometry group has maximal dimension. This extends to the Alexandrov setting the work carried out in \cite{CSX} for Riemannian manifolds with boundary. In our case, there appears a non-manifold with boundary and isometry group of maximal dimension, namely, the cone over a real projective space. Finally, in section~\ref{S:Sym_spaces}, we study symmetric Alexandrov spaces following previous work of   Berestovski{\u\i} \cite{Be}.


\begin{ack}
The first named author thanks the Department of Mathematics of the Universidad Aut\'onoma de Madrid for its hospitality while part of the work presented in this paper was carried out. Both authors would also like to thank Martin Weilandt and John Harvey for helpful comments on a first version of this paper.
\end{ack}


\section{Preliminaries}
\label{S:Prelims}

In this section we fix notation and collect some preliminary material that we will use in the rest of the paper. For the sake of completeness,  in subsection~\ref{SS:PIThm} we prove the Isotropy Lemma and the Principal Orbit Theorem for isometric Lie actions on Alexandrov spaces.  These results are well known in the smooth case (cf.~\cite{Gr2}).


\subsection{Background and notation}
Recall that a length space $(X,\mathrm{dist})$ of finite (Hausdorff) dimension is an \emph{Alexandrov space} (with curvature bounded below) if it has curvature bounded from below in the triangle comparison sense.
We refer the reader to \cite{BBI,BGP} for the main definitions and theorems pertaining to these spaces. Further developments can be found in \cite{Pet2}.

The primary source of technical difficulties in the study of Alexandrov spaces is usually due to the presence of singularities. However, given an Alexandrov space $X$, work of Otsu and Shioya \cite{OS}, and of Perelman \cite{Pe}, allows one to introduce a $C^0$ structure on a large open subset of $X$. More precisely,  call a point $p\in X$ \emph{regular}  if its space of directions $\Sigma_p$ is  isometric to the unit round sphere $\sphere^{n-1}$; otherwise, $p$ is called \emph{singular}. We will denote the set of regular points of $X$ by $R_X$, and the set of singular points by $S_X$, so that  $R_X=X\setminus S_X$.  The set $R_X$ is the intersection of countably many open and dense subsets of $X$  and is dense in $X$; any regular point has an open neighborhood $U$ and a map $\phi:U\to\RR^n$ that is a homeomorphism onto its image (cf.~\cite{OS,Pe}). In fact, there have been improvements on the smoothness of the structure (cf.~\cite{Ot}), although we shall omit them here, since they are not necessary in what follows.  Proving some of our statements in the Riemannian case requires using the differential of a map. Fortunately, this can also be defined for Alexandrov spaces (see, for example, \cite{Ly} for the more general statements).

Given  an Alexandrov space  $X$, we will denote its isometry group by $\Iso(X)$. The distance between two points $p,q\in X$  will be denoted by $|pq|$.  Given an isometric action $G\times X\to X$ of a group $G\leq \Iso(X)$,  we will denote the orbit of a point $p\in X$ by $G p$. The  isotropy group at $p$ will be denoted by $G_p$. 
Given a subset $A\subset X$, we will denote its image in $X/G$ under the orbit projection map $\pi:X\rightarrow X/G$ by $A^*$. Following this notation, we will denote the orbit space of the action $G\times X\rightarrow X$ by $X^*$, and a point in $X^*$ will be denoted by $p^*$, corresponding to the orbit $G p$. A useful fact about $\pi$ is that it is a \emph{submetry}, and as such admits horizontal lifts of geodesics from $X^*$ to $X$ (see, for instance, \cite{BeG}).

We will also use the fact that the natural action of $\Iso(X)$ on $X$ has closed orbits (see, for instance, \cite{KN}, section I.4). We will assume all actions to be effective and all Alexandrov spaces to be connected. We refer the reader to \cite{Br} for further background on the theory of transformation groups.


\subsection{Isotropy Lemma and Principal Orbit Theorem}
\label{SS:PIThm}

We now enunciate and prove Alexandrov versions of the Isotropy Lemma and the Principal Orbit  Theorem.  The main ideas follow  closely Lemma~1.3 and Theorem~1.4 in \cite{Gr2}, which are originally stated for the smooth case. We start with the  Isotropy Lemma, whose smooth version is due to Kleiner \cite{Kl}.


\begin{lem}[Isotropy Lemma] Let $G$ be a Lie group acting isometrically on an Alexandrov space $X$. If $c:[0,d]\rightarrow X$ is a minimal geodesic between the orbits $Gc(0)$ and $Gc(d)$, then, for any $t\in (0,d)$, the isotropy group $G_{c(t)}=G_c$ is a subgroup of $G_{c(0)}$ and of $G_{c(d)}$.
\end{lem}

\begin{proof}  Fix $t\in (0,d)$ and let $g\in G_{c(t)}$. Proceeding by contradiction, suppose that $g\notin G_{c(d)}$. The image of the minimal geodesic segment $c([t,d])$ under $g$ is a horizontal minimal geodesic segment from $gc(t)=c(t)$ to $gc(d)\neq c(d)$; it follows that the curve
$c[0,t]$ followed by $gc[t,d]$ still realizes the distance between the orbits $Gc(0)$ and $Gc(d)$, hence it is a geodesic. This is impossible, since we would have bifurcation of geodesics at $c(t)$
 (see Figure~\ref{F:isot_lemma}).  
\end{proof}


\begin{thm}[Principal Orbit  Theorem]  Let $G$ be a compact Lie group acting isometrically on an $n$-dimensional Alexandrov space $X$. Then there is a unique maximal orbit type and the orbits with maximal orbit type, the so-called principal orbits, form an open and dense  subset of $X$.
\end{thm}


\begin{proof}


\begin{figure}
\psfrag{A}{$c(0)$}
\psfrag{B}{$c(d)$}
\psfrag{C}{$c(t)$}
\psfrag{D}{$g c(0)$}
\psfrag{E}{$g c(d)$}
\psfrag{F}{$G c(0)$}
\psfrag{G}{$G c(d)$}
\includegraphics[scale=1.0]{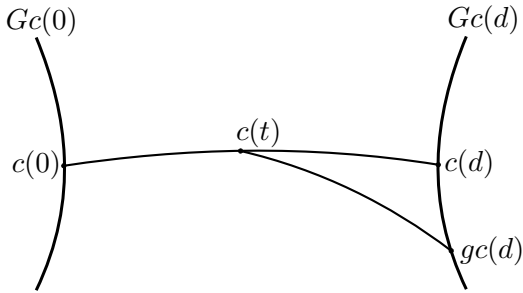}
\caption{Proof of the Isotropy Lemma}\label{F:isot_lemma}
\end{figure}

Among the isotropy groups of least dimension, let $H_0$ be one with the least number of connected components.
Such an $H_0$ exists because isotropy groups are compact. Let $p_0\in X$ be a point with isotropy group $H_0$. We claim that $H_0$ corresponds to a maximal orbit type. To see this, fix $p\in X$ distinct from $p_0$ and let $c:[0,1]\to X$ be a minimal geodesic between the orbits $G p_0$ and $G p$, with $c(0)=p_0$ and $c(1)=p$. Observe that the Isotropy Lemma and the choice of $H_0$ imply that $G_{c(t)}=H_0$ for $t\in [0,1)$. Thus, $H_0$ (or one of its conjugates) is a subgroup of  the isotropy group of any other point in $X$, and therefore the orbit through $p_0$ is a maximal orbit

Denote by $X_{(H_0)}$ the set of maximal orbits. Since any point $q\in X$ is the endpoint of a minimal geodesic between orbits starting at the orbit through $p_0$, it is clear that $X_{(H_0)}$ is dense in $X$. It remains to show that $X_{(H_0)}$ is open. This follows from the fact that, given an orbit $P\in X_{(H_0)}$, there exists a neighborhood of  $P$ such that for any other orbit $Q$ in this neighborhood $\mathrm{type}(Q)\geq \mathrm{type}(P)$ (cf.~\cite[Cor. II.5.5]{Br}). Since the orbit type of $P$ is maximal, the conclusion follows.

\end{proof}


\section{Dimension bound}
\label{S:Bounds}

We start by obtaining upper bounds for the dimension of the isometry group of an Alexandrov space.


\begin{thm}\label{dimension_bound}
Let $X$ be an Alexandrov space of dimension $n$. Then the dimension of its isometry group $G$ is at most $n(n+1)/2$. 
\end{thm}


\begin{proof}
We will use induction on the dimension of $X$. There is clearly nothing to prove in  the case $n=0$, so we will assume that the theorem holds for all positive integers less than $n$. 


\begin{lem}\label{L:identity} Let $X$ be a connected Alexandrov space and $f:X\rightarrow X$ an isometry. If there is some $p_0\in X$ such that $f(p_0)=p_0$ and $df_{p_{0}}:\Sigma_{p_0}\rightarrow \Sigma_{p_0}$ is the identity, then $f(p)=p$ for all $p$ in $X$.
\end{lem}

\begin{proof}[Proof of lemma~\ref{L:identity}]
It suffices to prove that the set of fixed points of $f$ is open in $X$. This follows from observing that $f(\exp_p(t v))=\exp_{f(p)}(t df_p(v))$, whenever both sides are defined. Since any point of $X$ can be connected to $p$ by a shortest geodesic, the set
\[
\{\, q\in X : q=\exp_p(tv)\text{ for some $v\in \Sigma_p$, $t\geq 0$}\,\}
\]
 contains an open neighborhood of $p$. Therefore, since $f(p)=p$ and $df_p=\Id$, the isometry $f$ must fix that same open neighborhood.  
\end{proof}

The main consequence of the lemma is that for any $p\in X$, the isotropy group $G_{p}$ acts faithfully by isometries on $\Sigma_{p}$.  It follows from the induction hypothesis that  $\dim G_{p}\leq \dim \Iso(\Sigma_{p})\leq n(n-1)/2$.

 Choose now a  regular point $p_0\in R_X$. We define the \emph{Myers-Steenrod map} $F:G\rightarrow X$ as  $F(g)=g\cdot p_0$. Clearly,  $F$ is continuous, because the topology of $G$ agrees with the compact-open topology; 
also, $F(G)\subset R_X$, since each isometry of $X$ maps regular points to regular points and singular points to singular points.

Let   $X^{ob}\supseteq R_X$ be  the set of points in $X$ for which there are at least $n$ directions making obtuse angles among them (cf.~\cite{Pe}). Fix some  $g_0\in G$ and take a chart $(U,\varphi)$ of $X^{ob}$ around $g_0\cdot p_0$. Observe that $\varphi:U\rightarrow \mathbb{R}^n$ is a homeomorphism on $R_X\cap U$ for the  structure introduced in \cite{OS} or \cite{Pe}.  
Now choose  open neighborhoods  $V$, $W$ around $g_0$ in $G$ such that $\overline{V}\subset W$ is compact  and with $\varphi\circ F$ defined on $W$. 
Since the quotient map $\pi:G\to G/G_{p_0}$ is a fiber bundle, we can assume, reducing the size of $W$ if necessary, that there is a section $s: \widetilde{W}\subseteq G/G_{p_0}\longrightarrow W$ with $\pi(W)=\widetilde{W}$ and $s([g_0])=g_0$.

Composing the maps $s$, $F$ and $\varphi$ results in a map $\tilde{F}:\widetilde{W}\longrightarrow \mathbb{R}^n$,  given explicitly by 
\[
\tilde{F}([g])=\varphi(s[g]\cdot p_0),
\]
which is clearly continuous and injective.

Finally, observe that $\pi(\overline{V})$ is compact,  $\tilde{F}\lvert_{\pi(\overline{V})}$ is continuous and injective, and $\mathbb{R}^n$ is Hausdorff. It follows that  the map
\[
\tilde{F}\lvert_{\pi(\overline{V})}:\pi(\overline{V})\rightarrow \tilde{F}(\pi(\overline{V}))\subseteq \mathbb{R}^n
\]
is a homeomorphism onto a subspace of $\mathbb{R}^n$. Hence $\dim \pi(\overline{V})\leq \dim\mathbb{R}^n=n$, and
\begin{equation}\label{dimension_sum}
\dim G=\dim \pi(\overline{V})+\dim G_{p_0} \leq n+\frac{n(n-1)}{2}=\frac{n(n+1)}{2}.
\end{equation}

\end{proof}


\section{Rigidity}
\label{S:Rigidity}
In this section we consider the case in which the dimension of the isometry group attains its maximal possible value.


\begin{thm}
\label{T:Rigidity}
Let $X$ be an $n$-dimensional Alexandrov space. If the dimension of the isometry group of $X$ is $n(n+1)/2$, then $X$ is isometric to one of the following space forms:
\vspace{.1in}
\begin{enumerate}
	\item An $n$-dimensional Euclidean space $\mathbb{R}^n$. \vspace{.1in}
	\item An $n$-dimensional sphere $\mathbb{S}^n$.\vspace{.1in}
	\item An $n$-dimensional projective space $\mathbb{R}P^n$.\vspace{.1in}
	\item An $n$-dimensional simply connected hyperbolic space $\mathbb{H}^n$.\vspace{.1in}
\end{enumerate}
\end{thm}


\begin{proof} 
It is enough to show that $X$ is isometric to a Riemannian manifold, since  the conclusion of the theorem holds in the Riemannian case (see, for instance, Theorem 3.1 in \cite{Ko}, p. 46). As in the the previous section, we proceed by induction, and we will continue to use the set up and notation introduced in the proof of Theorem~\ref{dimension_bound}. Clearly, the starting case is trivial and hence we assume that the theorem holds for every positive integer less than $n$.

Assume now that equality holds in equation~\ref{dimension_sum}, so that $\dim G_{p_0}=n(n-1)/2$ and $\dim \pi(\overline{V})=n$. Since $\Sigma_{p_0}$ is isometric to the unit round sphere $\mathbb{S}^{n-1}$, then $\dim G_{p_0}=n(n-1)/2$ implies that, up to finite factors, $G_{p_0}=\mathrm{SO}(n)$. We need now the following two lemmas: 


\begin{lem}\label{L:empty}
The set of singular points $S_X$ is empty. 
\end{lem}


\begin{proof}[Proof of Lemma \ref{L:empty}]
 Let $q$ be a point in $S_X$. Since $G_{p_0}$ acts transitively on $\Sigma_{p_0}$, the metric sphere $S_{p_0}(|p_0q|)$ is entirely composed of singular points. 
Let $p_1\not\in\overline{B_{p_0}(|p_0q|)}$ be a regular point. A geodesic from $p_0$ to $p_1$ must intersect $S_{p_0}(|p_0q|)$ at some singular point.  This contradicts the fact that regular points in an Alexandrov space form a convex set (see \cite{Pet}); hence $X=\overline{B_{p_0}(|p_0q|)}$. In fact, we could assume  $q$ to be the singular point closest to $p_0$ (and therefore we have shown that in our situation such singular points can not be dense). 
Since $p_0$ was an arbitrary regular point in $X$ and these form a dense subset of $X$, we conclude that $S_X$ is empty. 
\end{proof}


\begin{lem}\label{L:onto}
In the above situation, the Myers-Steenrod map $F:G\longrightarrow X$ is surjective.
\end{lem}


\begin{proof}[Proof of Lemma \ref{L:onto}]
We will first show that $F:G\longrightarrow R_X$ is an open map. Consider the following  commutative diagram:
\[
\xymatrix{
W\ar[r]^F & U \ar[d]^{\phi} \\
\pi(\overline{V}) \ar[r]^{\tilde{F}} \ar[u]^{s} & \RR^n
}
\]
where $U$, $V$ and $W$ are defined as in the proof of Theorem \ref{dimension_bound}.
We have $\dim \pi(\overline{V})=n$ and $\pi(\overline{V})$ has nonempty interior. Since $\tilde{F}$ is a homeomorphism onto its image, $\tilde{F}(\inte \pi(\overline{V}))$ is an open subset of dimension $n$ in $\tilde{F}(\pi(\overline{V}))\subseteq \RR^n$. Hence   $\img(\tilde{F})$ is open in $\RR^n$ and $\img (F) =G p_0$ is open in $X$.
Since $X$ is connected, and the orbits of $G$ give a partition of $X$, it follows that there can exist only one orbit. Therefore, $X=G p_0$. 
\end{proof}

Lemmas~\ref{L:empty} and \ref{L:onto} imply that $X$ is a homogeneous metric space with a lower curvature bound. It then follows from a result of  Berestovski{\u\i} \cite[Theorem~7]{Ber}  that $X$ is isometric to a Riemannian manifold, as claimed. This ends the proof of the main theorem in this section.

\end{proof}


\begin{rem}
To prove Theorem~\ref{T:Rigidity} one could also show first that $\Iso(X)$ acts transitively on the regular set $R_X$ and then use that the orbits of the action of the isometry group are closed sets. Nevertheless, the authors did not see a significant advantage in proceeding in this way. 
\end{rem}


\section{Dimension bound in the presence of extremal sets}
\label{S:bounds_extremal}
If an Alexandrov space contains extremal sets, then any isometry must preserve these, thus having restricted its possible image. As pointed out in the introduction, this has consequences for the dimension of the isometry group. We make this explicit in the next theorem.


\begin{thm}
\label{T:Bound_ext_set} Let $X$ be an $n$-dimensional Alexandrov space. If $X$ contains a $k$-dimensional connected extremal subset $E$, then the dimension of the isometry group of $X$ cannot exceed 
$$
{k+1 \choose 2} + {n-k \choose 2}.
$$
\end{thm}


\begin{proof}
We will use reverse induction on the dimension $k$. Naturally, if $k=0$,  then $E$ is an isolated point $p$. Any element in the connected component of the identity of $\Iso(X)$ must preserve $p$, and hence its differential induces an isometry on $\Sigma_p$. Since the action of $\Iso(X)_p$ is faithful on $\Sigma_p$, it follows from Theorem~\ref{dimension_bound} that $\dim\Iso(X)\leq n(n-1)/2$, as desired.

Assume now that the theorem is true for any Alexandrov space of dimension $n$ or less, and for all extremal subsets of dimension up to $k-1$. Any element in the connected component of the identity of $\Iso(X)$ must map $E$ to points of $E$, and thus the orbit of some $p\in E$ has dimension less than or equal to $k$. On the other hand, elements in $\Iso(X)_p$ act faithfully on $\Sigma_p$ and it is well known that the set of tangent directions to $E$ forms an extremal subset of $\Sigma_p$ of dimension $k-1$ (cf.~\cite{PP}). Hence, by the induction hypothesis,
$$
\dim \Iso_p\leq {k \choose 2} + {n-k \choose 2}
$$ 
and 
$$
\dim\Iso (X)\leq k+\dim\Iso_p(X)\leq {k+1 \choose 2} + {n-k \choose 2}.
$$
\end{proof}
 
The bounds given in the preceding theorem are optimal, as the following example shows. 


\begin{example}
Let $X$ be the Alexandrov space obtained from gluing in $\RR^{n+1}$ a cone with vertex $p_0$ with a hemisphere along their boundaries. Both spaces have nonnegative sectional curvature, thus $X$ has the same lower curvature bound. If the cone is taken with angle less than $\pi/2$, then its vertex $p_0$ is an extremal set in $X$. Since $X$ has spherical symmetry about its axis, its isometry group is $O(n)$, thus proving that the bound can not be improved when $k=0$. For higher $k$ take the product of this example with $\RR^k$, and let $E$ be $\{p_0\}\times \RR^k$. Clearly the isometry group of the whole space is $O(k)\times O(n)$, which corresponds to the bound given in Theorem~\ref{T:Bound_ext_set} once dimensions are readjusted.
\end{example}


\section{Dimension gaps}
\label{S:Gap_thms}

We generalize to Alexandrov spaces Theorems~3.2 and 3.3 in \cite[section II.3]{Ko}. The first theorem below shows that, as in the Riemannian case, there exists a gap in the possible dimensions of the isometry group of an Alexandrov space $X$ of dimension $n\neq 4$. The second theorem shows that $X$ is isometric to a Riemannian manifold when $\Iso(X)$ has the next possible largest dimension. As pointed out in \cite{Ko}, the techniques used in the proof of the Riemannian versions of these theorems do not work when $n=4$, due to the fact that $\mathrm{SO}(4)$ is not simple, and thus this case must be considered separately (cf.~\cite{Is}). 
Mann \cite{Ma} showed that, for Riemannian manifolds, there is a more general gap phenomenon which includes as a particular case Theorem~3.2 in \cite{Ko}. The third theorem in this section shows that this phenomenon also occurs for Alexandrov spaces.  


\begin{thm}
\label{T:gaps}
Let $X^n$ be an Alexandrov space of dimension $n\neq 4$. Then $\Iso(X)$ contains no closed subgroups of dimension $m$ for 
\[
\frac{1}{2}n(n-1)+1<m<\frac{1}{2}n(n+1).
\]
\end{thm}


\begin{proof}

 Let $G\leq \Iso(X)$ be a subgroup of dimension $m$, and let $G_p$ be the isotropy group of $G$ at a regular point $p\in R_X$. Observe first that the action of $G$ on $X$ induces an embedding $F:G/G_p \hookrightarrow X$, so that

\begin{equation}
\label{EQ:isotropy_ineq}
\dim G_p\geq \dim G -\dim X.
\end{equation}

Let $m>\frac{1}{2}n(n-1)+1$. Then, using inequality~\ref{EQ:isotropy_ineq}, we get
\begin{eqnarray*}
\dim G_p &
> & \frac{1}{2}n(n-1)+1-n\\[7pt]
& = & \frac{1}{2}(n-1)(n-2)+1.
\end{eqnarray*}
Since $\Sigma_p$ is isometric to the unit round sphere $\mathbb{S}^{n-1}$,  the isotropy group $G_p$ is a closed subgroup of $\mathrm{O}(n)$. It follows from the lemma on page 48 of \cite{Ko} that $G_p=\mathrm{SO}(n)$ or $\mathrm{O}(n)$, and therefore it acts transitively on $\Sigma_p$.

Recall that the set $R_X$ of regular points of $X$ is convex (see \cite{Pet}). Thus, given points $p,q\in R_X$  and a geodesic $\gamma$ between them, its midpoint $r$ is in $R_X$, and thus $G_r$ acts transitively on $\Sigma_r$. There is then an isometry fixing $r$ and mapping $\gamma$ to $\gamma^{\mathrm{op}}$ $(=\gamma(a-t))$, hence mapping $p$ to $q$ and vice versa. This implies that $\Iso(X)$ acts transitively on $R_X$. Since $R_X$ is dense in $X$ and the orbits of $\Iso(X)$ are closed in $X$, it follows that $X=R_X$ and $\Iso(X)$ acts transitively on $X$. Hence 
\begin{eqnarray*}
 m=\dim G & = &  \dim X + \dim G_p \\
 		 & = &  n+ \dim \mathrm{O}(n)\\
		 & = & \frac{1}{2}n(n+1).
\end{eqnarray*}
\end{proof}

We can also prove the following rigidity result. 


\begin{thm}
\label{rigidity}
Let  $X$ be an Alexandrov space of dimension  $n\neq 4$. If $\Iso(X)$ has dimension $\frac{1}{2}n(n-1)+1$, then $X$ is isometric to a Riemannian manifold, and thus is one of those listed in Theorem~3.3, p. 54 in \cite{Ko}.
\end{thm}


\begin{proof}
Since a homogeneous Alexandrov space must be isometric to a Riemannian manifold (cf.~\cite[theorem 7]{Ber}), it suffices to show that $\Iso(X)$ acts transitively on $X$.

 If $\Iso(X)$ does not act transitively on $X$, then 
\[
\dim \Iso(X)_p \geq \dim \Iso(X)-(n-1)=\frac{1}{2}(n-1)(n-2)+1.
\]
As in the proof of the previous theorem,  $\Iso(X)_p=\mathrm{O}(n)$ or $\mathrm{SO}(n)$,  $R_X=X$ and $\Iso(X)$ acts transitively on $X$. 
\end{proof}


The case $n=4$ is slightly different from the cases considered in Theorems~\ref{T:gaps} and \ref{rigidity} above, and we treat it separately. As pointed out at the beginning of this section, its peculiarity arises from the fact that $\mathrm{SO}(4)$ is not simple. The main tool used is Ishihara's paper \cite{Is}, where the Riemannian case was analyzed. For $n=4$, the statement of Theorem~\ref{T:gaps} changes to allow for the possibility of $8$-dimensional subgroups in $\Iso(X)$; its proof follows along the same lines as for general $n$, although one needs to use the lemma in  \cite[p.~347]{Is}, ruling out the existence of $5$-dimensional subgroups of $\mathrm{SO}(4)$. Thus, the dimension of $G_p$ must be $6$, and therefore it acts transitively on $\mathbb{S}^3$. To obtain the corresponding rigidity results one proceeds as in the proof of Theorem~\ref{rigidity}. Using the fact that $\mathrm{SO}(4)$ has no $5$-dimensional subgroups, it is easy to see that a $4$-dimensional Alexandrov space $X$ with  a group of isometries $G$ of dimension $7$ or $8$ must be a homogeneous space. Therefore, $X$ must be isometric to a homogeneous Riemannian manifold, and hence one of those considered by Ishihara in \cite{Is}. When $G$ is $7$-dimensional, this yields the analog of Theorem~\ref{rigidity} in  dimension $4$.  In the exceptional case, where $G=8$, the space $X$ must be isometric to a K\"ahler manifold of constant holomorphic sectional curvature (cf.~\cite[section 4]{Is}. These K\"ahler manifolds do not have higher dimensional analogues in the list of manifolds that occur for general $n$ in Theorem~\ref{rigidity}, in contrast to the $4$-dimensional spaces with a $7$-dimensional group of isometries.

We conclude this section with an extension to Alexandrov spaces of Mann's  gap theorem in \cite{Ma}. Observe that we recover Theorem~\ref{T:gaps} by setting $k=1$ in the theorem below.


\begin{thm}
Let $X$ be an $n$-dimensional Alexandrov space with $n\neq 4, 6, 10$. Then the isometry group $\Iso(X)$ contains no closed subgroup $G$ such that the dimension of $G$ falls into any of the ranges:
\\

\[
{n-k+1 \choose 2}+{k+1 \choose 2} < \dim G < {n-k+2 \choose 2}, \quad k\geq 1.
\]


\begin{proof}
Mann's arguments in \cite{Ma} carry over verbatim to our case, except for one detail:  In Cases A and B in the proof of  Theorem~1 in \cite{Ma}, one should choose the point $x\in X$ to be regular and lying in a principal orbit; this can be done, since the principal stratum is open and dense in $X$ and the set of regular points is dense in $X$ (cf.~section~\ref{S:Prelims}).
 
\end{proof}

\end{thm}


\section{ Isometry groups of Alexandrov spaces with boundary}
\label{S:Bdry}

We now extend to Alexandrov spaces  results by Chen, Shi and Xu \cite{CSX} for Riemannian manifolds with boundary. We follow their proof closely, noting that one must make appropriate modifications for it to work in the Alexandrov case.  We start with the following result.


\begin{thm}
\label{T:Dim_bound_bdry} If $X$ is an $n$-dimensional  Alexandrov space with non-empty boundary $\partial X$, then 
\[
\dim \Iso(X) \leq \frac{1}{2}(n-1)(n-2).
\]
\end{thm}


\begin{proof}

Since it is not known whether the boundary of an Alexandrov space is also an Alexandrov space, we must proceed with some care. We know that $\partial X$ is a metric space of Hausdorff dimension $n-1$. Given a point $p\in \partial X$, the orbit of $p$ under the action of $G=\Iso(X)$ is entirely contained in $\partial X$, and hence is of dimension at most $n-1$.    

Now, we proceed by induction, the case $n=1$ being trivial. Let $p\in \partial X$ and observe that the space of directions $\Sigma_p$ is an $(n-1)$-dimensional Alexandrov space with boundary. The isotropy $G_p$ acts on $\Sigma_p$ by isometries and, by the induction hypothesis, $\dim G_p\leq \frac{1}{2}(n-2)(n-3)$. The conclusion is immediate\end{proof}

As pointed out in the proof of the previous theorem, it is not known, in general, whether the boundary of an Alexandrov space is an Alexandrov space (with its induced intrinsic metric). The following proposition shows that, if the isometry group of an  Alexandrov space with boundary has maximal dimension, then the boundary is, in fact, a Riemannian manifold.


\begin{prop}
Let $X$ be an $n$-dimensional  Alexandrov space with non-empty boundary $\partial X$. If $\dim \Iso(X)=\frac{1}{2}(n-1)(n-2)$, then each connected component of $\partial X$, with its induced intrinsic metric, is isometric to one of the  space forms in Theorem~\ref{T:Rigidity}.
\end{prop}


\begin{proof}
Let $G=\Iso(X)$ and fix $p\in \partial X$. It follows from the proof of Theorem~\ref{T:Dim_bound_bdry} that, since the dimension of $G$ is maximal,  the dimension of the orbit $Gp\subset \partial X$ is $n-1$ and the dimension of the isotropy group $G_p$ is maximal. By the Theorem of Invariance of Domain, the orbit  $G p$ is open and, since it is also closed,   it must correspond to a connected component $B$ of $\partial X$. It follows from Theorem~7 in \cite{Ber} that $B$ is a Finsler homogeneous manifold. Since $G_p$ is of maximal dimension, the unit vectors at each point tangent to the boundary must be  a round sphere. Thus, the Finsler structure is Riemannian and the conclusion follows from Kobayashi's rigidity result \cite[Theorem 3.1]{Ko}.

\end{proof}

We conclude this section with the topological classification of Alexandrov spaces with boundary and isometry group of maximal dimension.


\begin{thm}
\label{T:Rigidity_bdry}
Let $X$ be an $n$-dimensional  Alexandrov space with non-empty boundary $\partial X$. If $\dim \Iso(X)=\frac{1}{2}(n-1)(n-2)$, then the following hold:
\vspace{.1in}
\begin{enumerate}
	\item If $X$ is compact, then it is  homeomorphic to one of the following spaces:\vspace{.1in}
	
		\begin{enumerate}
			\item A closed $n$-dimensional unit ball in $\RR^n$.\vspace{.1in}
			\item A cylinder $\mathbb{S}^{n-1}\times [0,1]$ or $\RR P^{n-1}\times [0,1]$.\vspace{.1in}
			\item $\RR P^{n}$ with a small open disc removed.\vspace{.1in}
			\item A cone over $\RR P^{n-1}$. \\ 
		\end{enumerate}
	\item If $X$ is non-compact and $\partial X$ has a compact component, then $X$ is  homeomorphic to a half open cylinder $\mathbb{S}^{n-1}\times [0,1)$ or $\RR P^{n-1}\times [0,1)$. \\ %
	\item If $X$ is non-compact and every component of $\partial X$ is non-compact, then $X$ is  homeomorphic to $\RR^{n-1}\times [0,1)$ or to $\RR^{n-1}\times [0,1]$.
\end{enumerate}

\end{thm}


\begin{rem}
The only new space in Theorem~\ref{T:Rigidity_bdry}  not appearing in the Riemannian case  is the cone over $\RR P^{n-1}$. 
\end{rem}


\begin{proof}[Proof of Theorem~\ref{T:Rigidity_bdry}]
By looking at level sets of the distance function to a component of  $\partial X$, we observe, as in \cite{CSX}, that the action of $\Iso(X)$ on $X$ is of cohomogeneity one. Thus the orbit space $X^*=X/\Iso(X)$ is homeomorphic  to either $[0,1]$ or   $[0,1)$.

 In case (1), the orbit space $X^*$ is homeomorphic to $[0,1]$. The isotropy at the endpoint of $X^*=[0,1]$ corresponding to a component of $\partial X$ is $\SO(n-1)$ or $\mathrm{O}(n-1)$, while the isotropy  at the orbits corresponding to points in the interior of $[0,1]$ is $\SO(n-1)$. The isotropy at the other endpoint must contain $\SO(n-1)$ and thus can only be $\mathrm{O}(n-1)$ or $\SO(n)$, yielding the desired spaces.  The only case not appearing in \cite{CSX} is the one in which the points $0$ and $1$ in $X^*\simeq [0,1]$ have, respectively, isotropy $\mathrm{O}(n-1)$ and $\SO(n)$. This corresponds to the cone over $\RR P^{n-1}$. Cases (2) and (3) follow as in \cite{CSX}.
\end{proof}


\begin{rem}  Since the action of $\Iso (X)$ on $X$ is of cohomogeneity one, the metric of $X$ is quite restricted. This idea was examined for smooth manifolds in \cite{CSX}, where the authors proved that the metrics were warped over an interval. It is natural to expect that similar versions can be given for Alexandrov spaces, although with certain modification on the warping function to reflect the possible lack of smoothness. The results in \cite{AB} can be used in this context to show that the metric is warped over an interval with a semiconcave function as warping factor. We have omitted the details for concision's sake, since they can be easily completed.
\end{rem}


\section{Symmetric and locally symmetric Alexandrov spaces}
\label{S:Sym_spaces}
Following \cite{Be}, we say that an Alexandrov space $X$ is a \emph{locally} \emph{(uniformly)} \emph{symmetric space} if for every $p\in X$ there is a neighborhood $U(p)$ of $p$ and a number $r>0$ such that for all $q\in U(p)$, the ball $B_r(q)$ admits an isometric involution with $q$ as its only fixed point. The space is \emph{symmetric} if the involutions extend to all the space. In the above reference, Berestovski{\u\i} shows that for simply connected $G$-spaces, locally symmetric and symmetric are equivalent conditions.
However, the situation is different for Alexandrov spaces, as we show in the following example.


\begin{example}
\label{E:LSYM_NOTSYM}
 Consider the double disk $X$. It is locally symmetric, simply connected and it is not symmetric. In fact, geodesic involutions do not extend to isometries of the whole space for points different of the centers of the disks, or of the glued boundary. More generally, if $\Omega\subset M^n$ is a geodesic ball in a symmetric space, then $\mathrm{Double}(\Omega)$ is locally symmetric but not symmetric.
\end{example}


\begin{example}
 Any regular polyhedral surface with an even number of edges meeting at each vertex, for example, the surface of an icosahedron, is locally symmetric but not symmetric, and is not the double of a geodesic ball in a symmetric space. The same holds for the double of a right cone over a circle of length less than $2\pi$. Thus, Example~\ref{E:LSYM_NOTSYM} does not exhaust all the possibilities for locally symmetric Alexandrov spaces that are not symmetric. The abundance of examples makes it difficult to give a general structure result for these spaces. 
\end{example}


\begin{example}
The double of the half-disc $D^2\cap \{\,(x,y)\in \mathbb{R}^2: y\geq 0\,\}$ in $\mathbb{R}^2$ is a non-negatively curved Alexandrov space that is neither symmetric nor locally symmetric, as can be seen by considering its vertex points.  This shows  that not every double of a convex domain in a symmetric space is locally symmetric. 
\end{example}

We conclude our paper with the main result of this section.


\begin{thm}
 A symmetric Alexandrov space is isometric to a Riemannian manifold.
\end{thm}


\begin{proof}
 Let $X^n$ be an $n$-dimensional Alexandrov space. Let $p\in X^n$ be a regular point, so that $\Sigma_p$ is isometric to the unit round sphere $\mathbb{S}^{n-1}$. Let $\sigma_p$ be the isometric involution at $p$, i.e., $\sigma_p^2=\mathrm{Id}$. Since $X^n$ is symmetric, $\sigma_p\in \mathrm{Iso}(X)$. We will now show that 
$\mathrm{d}\sigma_p:\Sigma_p \rightarrow \Sigma_p$ exists and agrees with $-\mathrm{Id}|_{\mathbb{S}^{n-1}}$.

Consider the sequence of isometries 
\[
 \sigma^a: (aX^n,p)\rightarrow (aX^n,p).
\]
Letting $a\rightarrow \infty$ we get an isometry
\[
\sigma^{\infty} : C_{o} (\mathbb{S}^{n-1})\simeq \mathbb{R}^n \rightarrow \mathbb{R}^n
\]
and $\mathrm{d}\sigma_p = \sigma^\infty |_{\mathbb{S}^{n-1}}$. Observe that $\mathrm{d}\sigma_p$ is $-\mathrm{Id}|_{\mathbb{S}^{n-1}}$, since $\sigma^\infty$ has no fixed points away from the origin and it is an orthogonal involution.

Let $x,y \in R_X$ and let $\gamma:[-a,a]\rightarrow X$ be a geodesic joining $x$ to $y$. Observe that the image of $\gamma$ is contained in $R_X$. Let $m=\gamma(0)$.  Since $\mathrm{d}\sigma_{m}=-\mathrm{Id}$, $\sigma_m$ sends $\exp_m(tv)$ to $\exp_m(-tv)$, we have that $\gamma(t)$ goes to $\gamma(-t)$. It follows that $\sigma_m$ preserves such geodesics, so $\mathrm{Iso}(X)$ acts transitively on $R_X$. Since $\Iso(X)$ has the compact-open topology, the orbit $\Iso(X)(p)$, $p\in R_X$, is closed and, containing all of $X\setminus S_X$, needs to be $X$. Therefore, $\Iso(X)$ acts transitively and $X$ is a homogeneous space. A result by Berestovski{\u\i} (cf.~\cite[Theorem~7]{Ber} applies to conclude that $X$ is isometric to a Riemannian manifold.

\end{proof}


\bibliographystyle{amsplain}


\end{document}